\documentclass[a4paper, leqno]{article}
\usepackage{amsmath}
\usepackage{amssymb}
\usepackage{amsthm}

\newcommand{\LSC}[1]{LSC(#1)}
\newcommand{\USC}[1]{USC(#1)}
\newcommand{\T}{\intercal}

\newcommand{\D}{\mathrm{D}}
\let\O\relax
\newcommand{\O}{\mathcal{O}(n)}

\newcommand{\<}{\left\langle}
\let\>\relax
\newcommand{\>}{\right\rangle}

\newcommand{\R}{\mathbb{R}}

\newcommand{\Sym}{\mathcal{S}}
\newcommand{\Symp}{\Sym^+}
\newcommand{\Sympp}{\Sym^	{++}}

\newcommand{\abs}[1]{\left\vert#1\right\vert}

\newcommand{\eps}{\varepsilon}
\newcommand{\kskip}{\vskip .0cm}

\DeclareMathOperator{\Id}{Id}

\DeclareMathOperator{\tr}{tr}

\newtheorem{definition}{Definition}[section]
\newtheorem{proposition}[definition]{Proposition}
\newtheorem{lemma}[definition]{Lemma}
\newtheorem{theorem}[definition]{Theorem}
\newtheorem*{acknowledgement}{Acknowledgement}

\newcommand{\cmin}{c_{\min}}
\newcommand{\cmax}{c_{\max}}

\newcommand{\cF}{c_F}
\newcommand{\cQ}{c_Q}

\newcommand{\Thetac}{\frac{\alpha+1}{\cF + \alpha (\cF - \cQ)}}

\newcommand{\lmin}{\lambda_{\min}}
\newcommand{\lmax}{\lambda_{\max}}

\author{Michael K\"uhn}
\title{Power- and Log-concavity of viscosity solutions to some elliptic Dirichlet problems}

\begin{document}
	\clearpage
	\maketitle
	\thispagestyle{empty}
	\setcounter{page}{0}
	
	\section{Introduction}
It has long been known that solutions $u$ to the torsion problem
\begin{align}	\label{eq:laplace_equation}
	- \Delta u	&= 1 \qquad\text{in } \Omega,	\\
			u	&= 0 \qquad\text{on } \partial \Omega,\notag
\end{align}
and positive first eigenfunctions of the Dirichlet-Laplacian
\begin{align}	\label{eq:eigen_laplace_equation}
	- \Delta u	&= \lambda u \phantom{0} \quad\text{in } \Omega,	\\
			u	&= 0 \phantom{\lambda u} \quad\text{on } \partial \Omega, \notag
\end{align}
in convex bounded domains $\Omega \subset \R^n$ have certain concavity properties. For $n = 1$ they are both concave, and for $n \geq 2$ the solution to \eqref{eq:laplace_equation} is $\frac{1}{2}$-power concave, i.e. $u^\frac{1}{2}$ is concave, while the solution to \eqref{eq:eigen_laplace_equation} is $\log$-concave.
We are interested in similar results for viscosity solutions to
\begin{align}	\label{eq:normalized_plaplace}
	-\abs{\nabla u}^\alpha \Delta_p^N u	&= 1	\qquad\text{in } \Omega, \\
	u &= 0 \qquad\text{on } \partial\Omega, \notag
\end{align}
and positive solutions of the eigenvalue problem
\begin{align}	\label{eq:eigen_normalized_plaplace}
	-\abs{\nabla u}^\alpha \Delta_p^N u	&= \lambda \abs{u}^\alpha u	\phantom{0}\quad\text{in } \Omega, \\
	u &= 0	\phantom{\lambda \abs{u}^\alpha u}\quad\text{on } \partial \Omega. \notag
\end{align}
Here $\alpha \in \R_{\geq 0}$, $p \in [2, \infty]$, $\Omega \subset \R^n$ is  a convex bounded domain that satisfies the interior sphere condition, and $\Delta_p^N$ denotes the normalized $p$-Laplacian operator as in \cite{kawohl2016geometry}, defined by
%
\[
	-\Delta_p^N u := -\frac{p-2}{p} \abs{\nabla u}^{-2} \< \nabla u, \D^2 u \nabla u \> - \frac{1}{p} \Delta u.
\]
More precisely we show that any positive viscosity solution $u$ to \eqref{eq:normalized_plaplace} with vanishing boundary values has the property that $u^\frac{\alpha+1}{\alpha+2}$ is a concave function.
Taking $\alpha = p-2$ this reproduces the result of \cite{sakaguchi1987concavity}, showing that weak solutions to
\[
	\begin{aligned}
		-\Delta_p u &=	1	&&\text{in } \Omega,	\\
				u	&=	0	&&\text{on } \partial \Omega
	\end{aligned}
\]
in the Sobolev sense are power concave with exponent $\frac{p-1}{p}$, in case of $p \geq 2$.
We are also able to reproduce power concavity for the normalized infinity Laplace operator shown in \cite{crasta2015dirichlet} with $\alpha = 2$, and \cite{crasta2016C1} with $\alpha = 0$. Thus we also extend the main result of \cite{kuehn2017thesis} where $\alpha = 0$ and $p \in [2, \infty]$ was discussed.
\\
The methods provided in \cite{alvarez1997convex} lay the foundation for this article.
The problem of power concavity itself was motivated by the pioneering work of \cite{makar1971solution}, \cite{korevaar1983capillary}, \cite{kawohl1984superharmonic}, \cite{kennington1985power}, and \cite{kawohl1985solutions}.
\\
Reults similar to ours have been shown in \cite{bianchini2013power} for very general operators and in \cite{kulczycki2017onconcavity} for the fractional Laplacian $\Delta^{\frac{1}{2}}$ in the planar case. Power concavity for parabolic equations was discussed in \cite{ishige2014note} and \cite{zhao2017power}.

To prove the desired power concavity, we use the comparison principle \cite[Theorem 1.3, Theorem 2.4]{lu2008uniqueness} and show that the convex envelope of the function $-u^\frac{\alpha+1}{\alpha+2}$ coincides with the function itself. For this it seems to be crucial that the convex envelope is not spanned by any boundary points. To assure this we first prove a Hopf-type Lemma.
\\

Secondly we
 show that any positive viscosity solution $u$ to \eqref{eq:eigen_normalized_plaplace} with vanishing boundary values has the property that $\log u$ is a concave function, given any positive eigenvalue $\lambda > 0$. For $p = 2$ and $\alpha = 0$ this goes back to \cite{brascamp1976extensions}, \cite{korevaar1983convex}, \cite{caffarelli1982convexity}, \cite{kawohl1985rearrangements}.
Again this reproduces the result of \cite{sakaguchi1987concavity} by choosing $\alpha = p-2$, $p \geq 2$. For this we only need to adapt the methods used in the first part. Though not necessary for obtaining the desired result we also show a Hopf-type Lemma for the Eigenvalue problem.

\subsection{Assumptions}
Some of our auxiliary results hold for more general equations. Therefore we consider first the equations
\begin{align}	\label{eq:main_equation}
	F(\nabla u, \D^2 u) &= f	\qquad\text{in } \Omega
	\intertext{and}
	\label{eq:eigen_main_equation}
	F(\nabla u, \D^2 u) &= f \abs{u}^\alpha u	\qquad\text{in } \Omega
\end{align}
with a continuous function $f : \Omega \to \R$ satisfying
\[
	\inf_\Omega f > 0,
\]
a convex domain $\Omega$ that satisfies the interior sphere condition, and a continuous operator $F: \R^n \setminus \{0\} \times \Sym \to \R$.
Here $\Sym \subset \R^{n \times n}$ denotes the set of real symmetric matrices.
We exclude all points $(0,X) \in \R^n \times \Sym$ because we also want to consider operators which are not well-defined but bounded at $(0,X)$.
In this case we use the lower and upper semicontinuous envelopes of $F$, defined by
\[
	F_*(q,X) := \lim_{r \downarrow 0} \inf \{F(z,X) ~\vert~ \abs{q-z} \leq r\}
\]
and
\[
	F^*(q,X) := \lim_{r \downarrow 0} \sup \{F(z,X) ~\vert~ \abs{q-z} \leq r\}
\]
respectively. At points $(q,X) \neq (0,X)$ or if $F$ is well-defined at $(q,X) = (0,X)$ we simply have $F(q,X) = F_*(q,X) = F^*(q,X)$.
\\

Furthermore we impose the following assumptions on $F$. For some constants $0 < \cmin \leq \cmax$ and all $a,b, \lambda \in \R$, $q \in \R^n \setminus \{0\}$, symmetric matrices $X,Y \in \Sym$, positive definite symmetric matrices $A_1, A_2 \in \Sympp$, and orthogonal $Q \in \O$ we assume $F$ to satisfy
\begin{equation} \tag{A1} \label{ass:linearity}
	F(q, a X + b Y) = a F(q, X) + b F(q, Y)
\end{equation}
\begin{equation} \tag{A2} \label{ass:boundedness}
	\cmin \abs{q}^\alpha \lmin(X)
	\leq - F(q, X)
	\leq \cmax \abs{q}^\alpha \lmax(X)
\end{equation}
\begin{equation} \tag{A3} \label{ass:scaling}
	F(\lambda q, X) = \abs{\lambda}^\alpha F(q,X)
\end{equation}
\begin{equation} \tag{A4} \label{ass:rotinvariance}
	F(q,X)	=	F(Q^\T q, Q^\T X Q)
\end{equation}
\begin{equation} \tag{A5}	\label{ass:operator_concavity}
	F^*(q,(\mu A_1 + (1-\mu)A_2)^{-1})
	\geq \mu F^*(q,A_1^{-1}) + (1-\mu) F^*(q,A_2^{-1}).
\end{equation}
The assumptions \eqref{ass:linearity} - \eqref{ass:rotinvariance} are easily verified when considering the operator appearing in \eqref{eq:normalized_plaplace}. In order to verify \eqref{ass:operator_concavity} one should make use of the fact that the mapping $(q,A) \mapsto \< q, A^{-1} q \>$ is convex, i.e.
\begin{align*}
	&\< t q_1 + (1-t) q_2, (t A_1 + (1-t) A_2)^{-1} (t q_1 + (1-t) q_2) \>	\\
	&\qquad\leq
	t \< q_1, A_1^{-1} q_1 \> + (1-t) \< q_2, A_2^{-1} q_2 \>
\end{align*}
for all $q_1,q_2 \in \R^n$, $A_1, A_2 \in \Sympp$ and $t \in [0,1]$. The proof can be found in \cite[p. 286]{alvarez1997convex}.
\\
We also want to mention that \eqref{ass:linearity} and \eqref{ass:boundedness} already imply \emph{degenerate ellipticity}, that is
\[
	F(q, X) \leq F(q, Y)
\]
whenever $X \geq Y$ for symmetric matrices $X,Y \in \Sym$.

\subsection{Viscosity solutions}
A suitable notion of weak solutions to problems like \eqref{eq:normalized_plaplace} was introduced in \cite{crandall1992user}. The so called \emph{viscosity solutions} can be used to investigate elliptic problems of second order with non-divergence structure. Since we will always assume vanishing boundary values in the classical sense, we only introduce the notion for equations without boundary conditions. For the definition we consider the more general equation
\begin{align}	\label{eq:general_main_equation}
	F(\nabla u(x), \D^2 u(x)) = g(x, u(x), \nabla u(x))
	\qquad\text{in } \Omega
\end{align}
with a function $g \in C(\Omega \times \R \times \R^n)$ that satisfies
\[
	-g(x,z_1, q) \leq -g(x,z_2, q)
\]
for all $x \in \Omega$, $z_1,z_2 \in \R, q \in \R^n$, whenever $z_1 \leq z_2$. One also says that $-g$ is \emph{proper}.
\begin{definition}[Viscosity solutions]
	A semicontinuous function $u \in \USC{\Omega}$ is a \emph{viscosity subsolution} to \eqref{eq:general_main_equation} if for every $\phi \in C^2(\Omega)$
	\[
		F_*(\nabla \phi(x_0), \D^2 \phi(x_0) \leq g(x_0,u(x_0),\nabla \phi(x_0))
	\]
	holds whenever $u-\phi$ attains a local maximum at a point $x_0 \in \Omega$ with $u(x_0) = \phi(x_0)$.	\\
	A semicontinuous function $u \in \LSC{\Omega}$ is a \emph{viscosity supersolution} to \eqref{eq:general_main_equation} if for every $\phi \in C^2(\Omega)$
	\[
		F^*(\nabla \phi(x_0), \D^2 \phi(x_0) \geq g(x_0,u(x_0), \nabla \phi(x_0))
	\]
	holds whenever $u-\phi$ attains a local minimum at a point $x_0 \in \Omega$ with $u(x_0) = \phi(x_0)$.	\\
	Finally a continuous function $u \in C(\Omega)$ 	is a \emph{viscosity solution} to \eqref{eq:general_main_equation} if it is both, a viscosity sub- and supersolution.
\end{definition}

An equivalent definition of viscosity solutions can be stated by using so called \emph{semijets} of second order. They resemble the set of derivatives appearing in  the Taylor approximation for smooth functions.
\begin{definition}[Semijets]
	For $u \in \USC{\overline{\Omega}}$ we define the \emph{second order superjet} of $u$ at $x \in \Omega$ by
	\begin{align*}
		J_\Omega^{2,+}u(x) :=
		\big\{
		(q,X) \in \R^n \times \Sym ~\big\vert~ u(y) \leq u(x) + \< q, y-x \> + \tfrac{1}{2} \< A(y-x),y-x\> 
		\\+\,o\big(\abs{y-x}^2\big) \text{ as } y \to x
		\big\}.
	\end{align*}
	For $u \in \LSC{\overline{\Omega}}$ we define the \emph{second order subjet} of $u$ at $x \in \Omega$ by
	\begin{align*}
	J_\Omega^{2,-}u(x) :=
	\big\{
	(q,X) \in \R^n \times \Sym ~\big\vert~ u(y) \geq u(x) + \< q, y-x \> + \tfrac{1}{2} \< A(y-x),y-x\>
	\\+\,o\big(\abs{y-x}^2\big) \text{ as } y \to x
	\big\}.
	\end{align*}
\end{definition}
\kskip

In order to consider viscosity solutions to boundary value problems, the authors of \cite{crandall1983viscosity} introduced the so called \emph{closure of semijets}. Though we do not rephrase the definition of viscosity solutions, they are necessary to use the results of \cite{alvarez1997convex}.

\kskip
\begin{definition}[Closure of semijets]
	For $u \in \USC{\overline{\Omega}}$ we define the \emph{second order superjet} of $u$ at $x \in \overline{\Omega}$ by
	\begin{align*}
		\overline{J}_{\overline{\Omega}}^{2,+}u(x) :=
		\big\{
		(q,X) \in \R^n \times \Sym ~\big\vert~ &\exists (x_n,u(x_n),q_n,X_n) \to (x,u(x),q,X)
		\\
		&\text{ as } n \to \infty \text{ with } (q_n,X_n) \in J_\Omega^{2,+}u(x_n)
		\big\}.
	\end{align*}
	For $u \in \LSC{\overline{\Omega}}$ we define the \emph{second order subjet} of $u$ at $x \in \overline{\Omega}$ by
	\begin{align*}
		\overline{J}_{\overline{\Omega}}^{2,-}u(x_n) :=
		\big\{
		(q,X) \in \R^n \times \Sym ~\big\vert~ &\exists (x_n,u(x_n),q_n,X_n) \to (x,u(x),q,X)
		\\
		&\text{ as } n \to \infty \text{ with } (q_n,X_n) \in J_\Omega^{2,-}u(x_n)
		\big\}.
	\end{align*}
\end{definition}
	\section{General results} \label{sec:general}
In this section we use radial solutions to \eqref{eq:main_equation} to derive a comparison principle with radial functions. This in turn is used to prove a Hopf-type Lemma for supersolutions to \eqref{eq:main_equation}. Therefore we first prove two simple but useful identities.
\begin{lemma}	\label{lem:c_F}
	There is a constant $\cF \in \R$ with
	\[
		F(q, \Id) = -\cF \abs{q}^\alpha.
	\]
	for any $q \in \R^n \setminus \{0\}$, satisfying $\cmin \leq \cF \leq \cmax$.
\end{lemma}
\begin{proof}
	Let $q_1,q_2 \in \R^n$ be arbitrary with $\abs{q_1} = \abs{q_2} = 1$. There is some $Q \in \O$ such that
	\[
		q_1 = Q^{-1} q_2 = Q^\T q_2.
	\]
	Then, by \eqref{ass:rotinvariance},
	\[
		F(q_1, \Id) = F(Q^\T q_2, \Id) = F(Q^\T q_2, Q^\T Q)
		= F(q_2, \Id).
	\]
	So there must be a constant $\cF$ such that
	\[
		F(q, \Id) = -\cF
	\]
	for any $q$ with $\abs{q} = 1$. Taking any other $q \in \R^n \setminus \{0\}$ we find
	\[
		F(q, \Id) = -\cF \abs{q}^\alpha,
	\]
	using \eqref{ass:scaling}.
	The estimate follows from \eqref{ass:boundedness}.
\end{proof}

\kskip

\begin{lemma}	\label{lem:c_Q}
	There is a constant $\cQ \in \R$ with
	\[
		F(q, q \otimes q) = - \cQ \abs{q}^{\alpha+2}
	\]
	for any $q \in \R^n \setminus \{0\}$, satisfying $0 \leq \cQ \leq \cF$.
\end{lemma}
\begin{proof}
	Let $q = (q_1,\ldots,q_n)^\T \in \R^n \setminus \{ 0 \}$. First we note that
	\begin{align*}
		q \otimes q
			=	[q_1 \, q, \ldots q_n \, q],
	\end{align*}
	so all $n$ columns are linear dependent. Thus the matrix has $n-1$ times the eigenvalue $0$. Also we note that $q$ itself is an eigenvector with corresponding eigenvalue $\abs{q}^2$. So let $Q \in \O$ be a diagonalizing matrix in a way that
	\[
		[Q^\T (q \otimes q) Q]_{_{1,1}} = \abs{q}^2
	\]
	and
	\[
		[Q^\T (q \otimes q) Q]_{_{i,j}} = 0
	\]
	for $2 \leq i,j \leq n$. We compute
	\begin{align*}
		[Q^\T (q \otimes q) Q]_{_{i,j}}
			=	\sum_{k=1}^n Q^\T_{i,k} \sum_{l=1}^n q_k q_l \, Q_{l,j}
			=	\sum_{k=1}^n Q^\T_{i,k} \, q_k \sum_{l=1}^n Q^\T_{j,l} \, q_l
			=	[Q^\T \, q]_{_{i}} [Q^\T \, q]_{_{j}}
	\end{align*}
	to find that
	\begin{align*}
		[Q^\T \, q]_{_{1}}^2
			=	[Q^\T (q \otimes q) Q]_{_{1,1}}
			=	\abs{q}^2
	\end{align*}
	and
	\begin{align*}
	[Q^\T \, q]_{_{i}}^2
			=	[Q^\T (q \otimes q) Q]_{_{i,i}}
			=	0.
	\end{align*}
	for $2 \leq i \leq n$. So $F(q, q \otimes q) = F(Q^\T q, Q^\T (q \otimes q) Q)$ only depends on $\abs{q}$.
	We conclude that there is a function $h : \R \to \R$ such that
	\[
		F(q, q \otimes q) = h(\abs{q}).
	\]
	We deduce from \eqref{ass:linearity} and \eqref{ass:scaling} that there must be a constant $\cQ$ such that
	\[
		h(\abs{q})
			=	\abs{q}^{\alpha+2} h(1)
			=	- \cQ \abs{q}^{\alpha+2}.
	\]
	For the estimate we see
	\[
		0
			\leq	- F(q,q \otimes q)
			=		\cQ \abs{q}^{\alpha+2}
	\]
	and
	\[
		\cQ \abs{q}^{\alpha+2}
			=		- F(q,q \otimes q)
			\leq	- F(q, \abs{q}^2 \Id)
			=		\cF \abs{q}^{\alpha+2},
	\]
	using Lemma \ref{lem:\cF}.
\end{proof}
\kskip


Using these identities we may compute $F*$ for an arbitrary smooth radial function $(r \mapsto \phi(r)) \in C^2(0,\infty)$ with $r = \abs{x}$ and $\nabla \phi(r) \neq 0$ to obtain
\begin{equation}	\label{eq:cusp_derivation}
	F^*( \nabla \phi(r), \D^2 \phi(r) )
	=	-\abs{\phi'(r)}^\alpha
	\left(
		\phi''(r)	\cQ + \frac{1}{r} \phi'(r) (\cF - \cQ)
	\right).
\end{equation}
Therefore the equation
\[
	F^*( \nabla \phi(\abs{x}), \D^2 \phi(\abs{x}) ) = K	\qquad\text{in } \Omega\setminus\{0\}
\]
for any constant $K \geq 0$ is solved by
\begin{align*}
	\phi(r) :=
		a - \frac{\alpha+1}{\alpha+2}
		\left( \Thetac K \right)^\frac{1}{\alpha+1}
		r^\frac{\alpha+2}{\alpha+1}
\end{align*}
with an arbitrary constant $a \in \R$.
\kskip

\begin{proposition}	\label{prop:cusp_solution}
	Let $U \subset \Omega$ be bounded, $x_0 \in U$, and $K \geq 0$ be any nonnegative constant. Furthermore let $a \in \R$ be any constant. Then the function
	\[
		\Phi(x) :=
			a - \frac{\alpha+1}{\alpha+2}
			\left( \Thetac K \right)^\frac{1}{\alpha+1}
			\abs{x-x_0}^\frac{\alpha+2}{\alpha+1}
	\]
	satisfies
	\[
		F^*( \nabla \Phi, \D^2 \Phi) = K
	\]
	in $U \setminus \{x_0\}$.
\end{proposition}
\begin{proof}
	This is a straightforward computation.
\end{proof}
\kskip
When taking $K = \inf_U f - \eps \geq 0$ for $\eps > 0$ suitable small, the function $\Phi$ just breaks the definition of supersolutions. So for any supersolution $u \in \LSC{\overline{\Omega}}$, the difference $u-\Phi$ cannot attain any local minima which is the idea for the following comparison principle. It extends the \emph{comparison with cones} \cite{evans2011everywhere} and \emph{comparison with polar quadratic polynomials} \cite{lu2008pde} principles.
\kskip
\begin{lemma}	\label{lem:comparison_with_radial}
	Let $U \subset \Omega$ be bounded, $x_0 \in U$ and $u \in \LSC{\overline{\Omega}}$ be a supersolution of \eqref{eq:main_equation}.
	Then
	\begin{align*}
		u(x) \geq \Phi(x)	\qquad\text{on } \partial(U \setminus\{x_0\})
		&&\implies&&
		u(x) \geq \Phi(x)	\qquad\text{in } U
	\end{align*}
	with
	\[
		\Phi(x) :=
			a - \frac{\alpha+1}{\alpha+2}
			\left( \Thetac \inf_U f \right)^\frac{1}{\alpha+1}
			\abs{x-x_0}^\frac{\alpha+2}{\alpha+1},
	\]
	and any constant $a \in \R$.
\end{lemma}
\begin{proof}
	We assume there is some $\hat x \in U$ such that
	\begin{align*}
		u(x) \geq \Phi(x) \qquad\text{on } \partial(U \setminus\{x_0\})
		&&\text{but}&&
		u(\hat x) < \Phi(\hat x).
	\end{align*}
	Since $U$ is bounded, there is some ball of radius $R$ covering $U$. Then the function
	\begin{align*}
		\Phi_\eps(x) :=
		a &- \frac{\alpha+1}{\alpha+2}
		\left( \Thetac \left(\inf_U f - \eps\right) \right)^\frac{1}{\alpha+1}
		\abs{x-x_0}^\frac{\alpha+2}{\alpha+1}	\\
		&- \frac{\alpha+1}{\alpha+2}
		\left( \Thetac \eps \right)^\frac{1}{\alpha+1}
		R^\frac{\alpha+2}{\alpha+1}
	\end{align*}
	satisfies
	\[
		u(x) \geq \Phi(x) \geq \Phi_\eps(x) \qquad\text{on } \partial(U \setminus\{x_0\})
	\]
	while still keeping
	\[
		u(\hat x) < \Phi_\eps(\hat x)
	\]
	and $\inf_U f - \eps \geq  0$ for $\eps > 0$ sufficiently small. Therefore we may assume that $x \mapsto u(x) - \Phi_\eps(x)$ attains a local minimum at $\hat x$. There
	\[
		F^*( \nabla \Phi_\eps(\hat x), \D^2 \Phi_\eps(\hat x)) = \inf_U f - \eps < f(\hat x)
	\]
	holds, contradicting the assumption on $u$ being a supersolution of \eqref{eq:main_equation}.
\end{proof}
\kskip
The advantage of comparing supersolutions to an explicit function is that we also obtain a Hopf-type result.
\kskip
\begin{lemma}[Hopf's lemma]
	\label{lem:hopf}
	Let $u \in \LSC{\overline{\Omega}}$ be a supersolution of \eqref{eq:main_equation}, $\Omega$ satisfy the interior sphere condition, and let $x_0 \in \partial \Omega$ be a boundary point with
	\begin{equation}	\label{eq:hopf_assumption}
		u(x_0) < u(x)
	\end{equation}
	for all $x \in \Omega$. Then
	\[
		\limsup_{r \downarrow 0} \frac{u(x_0) - u(x_0 - r \mu)}{r} < 0
	\]
	for any $\mu \in \R^n$ with $\< \mu, \nu(x_0) \> > 0$
	with $\nu(x_0)$ denoting the outer normal at $x_0$.
\end{lemma}
\begin{proof}
	Since $\Omega$ satisfies the interior sphere condition there is some $R > 0$ such that $B_R(y_0) \subset \Omega$ and $\partial B_R(y_0) \cap \partial \Omega = \{x_0\}$ with $y_0 := x_0 - R \nu(x_0)$.
	Since $f$ is positive, $u$ is also a supersolution of
	\[
		F(\nabla u, \D^2 u) = \eps \inf_\Omega f
	\]
	in $\Omega$ for all $\eps \in (0,1)$. We may take $\eps > 0$ so small that
	\begin{equation}	\label{eq:hopf_equation1}
		u(y_0) \geq
		u(x_0) + \frac{\alpha+1}{\alpha+2}
		\left( \Thetac \left(\eps \inf_{B_R(y_0)} f \right) \right)^\frac{1}{\alpha+1}
		R^\frac{\alpha+2}{\alpha+1}.
	\end{equation}
	Then we define
	\begin{align*}
		\psi(x) :=
		u(x_0) + \frac{\alpha+1}{\alpha+2}
		\left( \Thetac \left(\eps \inf_{B_R(y_0)} f \right) \right)^\frac{1}{\alpha+1}\times \\
		\left(R^\frac{\alpha+2}{\alpha+1} - \abs{x-y_0}^\frac{\alpha+2}{\alpha+1}\right).
	\end{align*}
	Now we have
	\[
		u(y_0) \geq \psi(y_0)
	\]
	by \eqref{eq:hopf_equation1} and
	\[
		u(x) > u(x_0) = \psi(x)
	\]
	on $\partial B_R(y_0)$ by \eqref{eq:hopf_assumption}. Therefore we have $u \geq \psi$ on $\partial (B_R(y_0) \setminus \{y_0\})$.
	Then Lemma \ref{lem:comparison_with_radial} implies $u \geq \psi$ in $B_R(y_0)$.
	\\
	
	Using $u(x_0) = \psi(x_0)$ and $x_0 - r \mu \in B_R(y_0)$ we conclude
	\[
		\frac{u(x_0) - u(x_0 - r \mu)}{r} \leq \frac{\psi(x_0) - \psi(x_0 - r \mu)}{r}
	\]
	for all $r > 0$ sufficiently small. Sending $r \downarrow 0$ we find
	\begin{align*}
		&\limsup_{r \downarrow 0} \frac{u(x_0) - u(x_0 - r \mu)}{r}	\\
		&\qquad\leq	\frac{\partial \psi}{\partial \mu}(x_0)	\\
		&\qquad=	-\left( \Thetac \left(\eps \inf_{B_R(y_0)} f\right) \, R \right)^\frac{1}{\alpha+1}	\< \nu(x_0), \mu \>\\
		&\qquad<	0,
	\end{align*}
	proving the assertion.
\end{proof}
	\section{Power concavity}
In this section we consider the special case of $f \equiv 1$ and show that any solution $u$ to
\begin{align}	\label{eq:main_equation_constant}
	F(\nabla u, \D^2 u) &= 1	\qquad\text{in } \Omega
\end{align}
with vanishing boundary values has the property that
\begin{equation}
	w := -u^\frac{\alpha+1}{\alpha+2}
\end{equation}
is a convex function. Therefore we first consider which equation $w$ must solve if $u$ is a solution to \eqref{eq:main_equation_constant}.

\begin{lemma}	\label{lem:transformed_equation}
	A function $u \in \USC{\Omega}$ is a positive subsolution to \eqref{eq:main_equation_constant} if and only if  $w 
	\in \LSC{\Omega}$ is a negative supersolution to
	\begin{align}	\label{eq:main_transformed}
		F(\nabla w, \D^2 w)
		=	\frac{1}{w} \left( \frac{\cQ}{\alpha+1} \abs{\nabla w}^{\alpha+2} + \left(\frac{\alpha+1}{\alpha+2}\right)^{\alpha+1} \right).
	\end{align}
\end{lemma}
\begin{proof}
	The proof is a straightforward computation which relies on \eqref{ass:linearity} and Lemma \ref{lem:c_Q}.
\end{proof}
\kskip

In the second step we make sure that the convex envelope of $w$, defined by
\begin{align*}
	w_{**}(x) := \sup \big\{ \sum_{i = 1}^{k} \mu_i w(x_i)
	~\vert~
	&x = \sum_{i = 1}^{k} \mu_i x_i,\\
	&x_i \in \overline{\Omega},
	0 \leq \mu_i \leq 1,
	\sum_{i=1}^{k} \mu_i = 1,
	k \leq n+1 \big\},
\end{align*}
is not spanned by any boundary points.
\kskip

\begin{lemma}	\label{lem:convex_decomp}
	Let $u \in \LSC{\overline{\Omega}}$ be a positive supersolution to \eqref{eq:main_equation_constant}.
	Furthermore let $x \in \Omega$, $x_1, \ldots x_k \in \overline{\Omega}$, $\sum_{i=1}^{k} \mu_i = 1$ with
	\begin{align*}
		x = \sum_{i=1}^{k} \mu_i x_i
		&&\text{and}&&
		w_{**}(x) = \sum_{i=1}^{k} \mu_i w(x_i).
	\end{align*}
	Then $x_1, \ldots, x_k \in \Omega$.
\end{lemma}
\begin{proof}
	We assume this is not true. So without loss of generality let $x_1 \in \partial \Omega$.
	There must be at least one $x_i \not\in \partial\Omega$ because otherwise we would have
	\[
		0 > w(x) \geq w_{**}(x) = 0.
	\]
	So, again without loss of generality, we may assume $x_2 \in \Omega$. Then, by definition, $w_{**}$ is affine on the segment $[x_1,x_2]$.
	So there is some finite constant $c > 0$ such that
	\begin{align}	\label{eq:proof_conv_int1}
		\tfrac{w_{**}(x_1) - w_{**}(x_1 - t(x_2-x_1))}{t}
	= c
	\end{align}
	for all $t \in (0,1)$.
	\\
	On the other hand we may use Young's inequality to find
	\[
		u(x_1)^\frac{\alpha+1}{\alpha+2} \, u(x_1 - t (x_2-x_1))^{1-\tfrac{\alpha+1}{\alpha+2}}
		\leq \tfrac{\alpha+1}{\alpha+2} u(x_1) + \left(1-\tfrac{\alpha+1}{\alpha+2}\right)u(x_1 - t (x_2-x_1)).
	\]
	By rearranging terms and using the definition of $w$ we obtain
	\begin{align}	\label{eq:proof_conv_int2}
		\tfrac{w(x_1) - w(x_1 - t (x_2-x_1))}{t}
		\geq	- \tfrac{\alpha+1}{\alpha+2} \tfrac{1}{u(x_1 - t (x_2-x_1))^\frac{1}{\alpha+2}} \tfrac{u(x_1)-u(x_1-t(x_2-x_1))}{t}.
	\end{align}
	Since $u$ is positive in $\Omega$ we have $u(x_1 - t (x_2-x_1)) \downarrow 0 $ as $t \downarrow 0$. Furthermore since $\Omega$ is convex we also have $\< x_2 - x_1, \nu(x_1) \> > 0$ and may invoke Hopf's Lemma \ref{lem:hopf} to obtain
	\[
		\limsup_{t \downarrow 0} \tfrac{u(x_1) - u(x_1 - t(x_2-x_1))}{t} < 0.
	\]
	This together with \eqref{eq:proof_conv_int2} implies
	\[
		\tfrac{w(x_1) - w(x_1 - t (x_2-x_1))}{t} \to \infty
	\]
	as $t \downarrow 0$, contradicting \eqref{eq:proof_conv_int1} since $w_{**}(x_1) = w(x_1) = 0$ and $-w_{**} \geq -w$.
\end{proof}
\kskip

In the third step we show that the convex envelope $w_{**}$ is a supersolution. We first prove some technical consideration and proceed to the more difficult proof right after.
\begin{lemma} \label{lem:operator_concavity}
	For every $q \in \R^n$ the mapping $A \mapsto \frac{1}{-F^*(q,A^{-1})}$ is concave in $\Sympp$.
\end{lemma}
\begin{proof}
	The proof is very similar to \cite[p.287]{alvarez1997convex}, showing the concavity of the mapping
	\[
		A \mapsto \frac{1}{\tr A^{-1}}.
	\]
	It can be adapted by replacing $\tr A^{-1}$ with $-F^*(q,A^{-1})$ and using \eqref{ass:operator_concavity}, so we omit the proof.
\end{proof}
\kskip

\begin{lemma}	\label{lem:convex_also_solution}
	Let $u \in \USC{\overline{\Omega}}$ be a positive subsolution to \eqref{eq:main_equation_constant} with $u = 0$ on $\partial \Omega$. Then
	$w_{**}$ is a supersolution to \eqref{eq:main_transformed} with $w_{**} = 0$ on $\partial \Omega$.
\end{lemma}
\begin{proof}
	According to \cite[Lemma 1]{alvarez1997convex} we have $w_{**} = w = 0$ on $\partial \Omega$. So we only have to show that $w_{**}$ is indeed a supersolution to \eqref{eq:main_transformed}. This means showing that
	\[
		F^*(q, X) - \frac{1}{w_{**}(x)} \left( \frac{\cQ}{\alpha+1} \abs{q}^{\alpha+2} + \left(\frac{\alpha+1}{\alpha+2}\right)^{\alpha+1} \right)
		\geq	0
	\]
	for all $x \in \Omega$ and $(q,X) \in J_{\overline{\Omega}}^{2,-}w_{**}(x)$.\\
	Using \cite[Lemma 3]{alvarez1997convex} we only have to consider the case of $X \geq 0$.
	\\
	
	If $q = 0$ and $\alpha > 0$ we have
	\[
		F^*(q,X) = 0
	\]
	by \eqref{ass:boundedness}. So the assertion is obvious since $w_{**} \leq 0$ in $\overline{\Omega}$. So let us assume $q \neq 0$ or $\alpha = 0$.
	\\
	
	By Lemma \ref{lem:convex_decomp} we can decompose $x$ in a convex combination of interior points $x_1, \ldots, x_k \in \Omega$ such that
	\begin{align*}
		\sum_{i=1}^{k} \mu_i x_i = x
		&&\text{and}&&
		\sum_{i=1}^{k} \mu_i w(x_i) = w_{**}(x)
	\end{align*}
	for some $\mu_i, \ldots, \mu_k > 0$ with $\sum_{i=1}^{k} \mu_i = 1$.
	It is interesting to note that for $C^1$-solutions $\nabla w_{**}(x_i)$ is independent of $i \in \{1, \ldots, k\}$. Therefore the proof works for quasilinear equations. A similar effect happens in the classical proof in \cite[p. 117]{kawohl1985rearrangements}.
	Then by \cite[Proposition 1]{alvarez1997convex} for every $\eps > 0$ small enough, there are positive semidefinite matrices $X_1,\ldots,X_k \in \Symp$ such that
	\[
		X - \eps X^2
		\leq \left( \sum_{i=1}^{k} \mu_i X_i^{-1}\right)
		=: Y.
	\]
	We may assume that the matrices $X_1, \ldots, X_k$ are positive definite. Otherwise we consider matrices
	\[
		\tilde X_i := X_i + \tfrac{1}{n} \Id
	\]
	and afterwards take the limit $n \to \infty$ as in
	\cite[p. 273]{alvarez1997convex}.
	Using that by assumption $w$ is a supersolution to \eqref{eq:main_transformed}, we find
	\[
		F^*(q, X_i) \geq
		\frac{1}{w(x_i)}
		\left( \frac{\cQ}{\alpha+1} \abs{q}^{\alpha+2} + \left(\frac{\alpha+1}{\alpha+2}\right)^{\alpha+1} \right)
	\]
	for $i = 1,\ldots, k$.
	Since $X_i$ is positive definite and we are considering the case of $q \neq 0$ or $\alpha =0$, we have $F^*(q,X_i) < 0$ according to \eqref{ass:boundedness} which allows us, together with $w < 0$, to rearrange the preceding inequality to
	\[
		-w(x_i)
		\leq - \frac{1}{F^*(q, X_i)} \left( \frac{\cQ}{\alpha+1} \abs{q}^{\alpha+2} + \left(\frac{\alpha+1}{\alpha+2}\right)^{\alpha+1} \right)
	\]
	for $i = 1,\ldots, k$. By summation and rearranging terms again, we obtain
	\[
		-\frac{1}{\sum_{i=1}^{k} \mu_i w(x_i)} \left( \frac{\cQ}{\alpha+1} \abs{q}^{\alpha+2} + \left(\frac{\alpha+1}{\alpha+2}\right)^{\alpha+1} \right) \geq \left( \sum_{i=1}^{k} \mu_i \frac{1}{- F^*(q,X_i)} \right)^{-1}.
	\]
	Using that $F^*$ is degenerate elliptic and plugging in $X - \eps X^2 \leq Y$ we find
	\begin{align*}
		&F^*(q, X - \eps X^2) - \frac{1}{w_{**}(x)} \left( \frac{\cQ}{\alpha+1} \abs{q}^{\alpha+2} + \left(\frac{\alpha+1}{\alpha+2}\right)^{\alpha+1} \right)	\\
		&\qquad\geq F^*(q,Y) - \frac{1}{w_{**}(x)} \left( \frac{\cQ}{\alpha+1} \abs{q}^{\alpha+2} + \left(\frac{\alpha+1}{\alpha+2}\right)^{\alpha+1} \right) \\
		&\qquad= F^*(q,Y) - \frac{1}{\sum_{i=1}^{k} \mu_i w(x_i)} \left( \frac{\cQ}{\alpha+1} \abs{q}^{\alpha+2} + \left(\frac{\alpha+1}{\alpha+2}\right)^{\alpha+1} \right)	\\
		&\qquad\geq F^*(q,Y) + \left( \sum_{i=1}^{k} \mu_i \frac{1}{-F^*(q,X_i)}\right)^{-1}.
	\end{align*}
	By Lemma \ref{lem:operator_concavity} we have
	\begin{align*}
		\sum_{i=1}^{k} \mu_i \frac{1}{-F^*(q,X_i)}
		&= \sum_{i=1}^{k} \mu_i \frac{1}{-F^*(q,(X_i^{-1})^{-1})}	\\
		&\leq \frac{1}{-F^*(q, \sum_{i=1}^k \mu_i X_i^{-1})}	\\
		&= \frac{1}{-F^*(q, Y)}
	\end{align*}
	so we may combine these two estimates to obtain
	\[
		F^*(q, X - \eps X^2) \geq 0.
	\]
	Sending $\eps \downarrow 0$ concludes the assertion.
\end{proof}
\kskip

Finally we are able to present our first main result. In this last step we use comparison principles to show that $w_{**}$, being a supersolution, must be equal to $w$ itself.

\kskip
\begin{theorem}	\label{th:concavity}
	Let $u \in C(\Omega)$ be a positive viscosity solution to \eqref{eq:main_equation_constant} with $u = 0$ on $\partial \Omega$ in a convex domain $\Omega$ that satisfies the interior sphere condition. Furthermore we assume that \eqref{ass:linearity}, \eqref{ass:boundedness}, \eqref{ass:scaling}, \eqref{ass:rotinvariance}, and \eqref{ass:operator_concavity} hold.
	Then $u^\frac{\alpha+1}{\alpha+2}$ is concave.
\end{theorem}

\begin{proof}
	By definition $u$ is both, a sub- and a supersolution to \eqref{eq:main_equation_constant}. Then, by Lemma \ref{lem:transformed_equation}, we find that $w := -u^\frac{\alpha+1}{\alpha+2}$ is a negative supersolution to \eqref{eq:main_transformed}. Using Lemma \ref{lem:convex_also_solution}, we obtain that $w_{**} \leq w$ is also a negative supersolution to \eqref{eq:main_transformed}.
	Then, again by Proposition \ref{lem:transformed_equation}, $(-w_{**})^\frac{\alpha+2}{\alpha+1}$ is a positive subsolution to \eqref{eq:main_equation_constant}. Invoking the comparison principle \cite[Theorem 1.3, Theorem 2.4]{lu2008uniqueness} we find
	\[
		(-w_{**})^\frac{\alpha+2}{\alpha+1}
		\leq	u
		=		(-w)^\frac{\alpha+2}{\alpha+1}.
	\]
	On the other hand we have $w_{**} \leq w \leq 0$, so $(-w_{**})^\frac{\alpha+2}{\alpha+1} \geq (-w)^\frac{\alpha+2}{\alpha+1}$ and finally
	\[
		(-w_{**})^\frac{\alpha+2}{\alpha+1} = (-w)^\frac{\alpha+2}{\alpha+1}.
	\]
	We may conclude that $w_{**} = w$, making $w$ a convex and $u^\frac{\alpha+1}{\alpha+2}$ a concave function.
\end{proof}
	\section{Log concavity}
In Section \ref*{sec:general} we have seen that we can compare supersolutions to \eqref{eq:main_equation} with functions
\begin{align*}
	\Phi(x) :=
		a - \frac{\alpha+1}{\alpha+2}
		\left( \Thetac K \right)^\frac{1}{\alpha+1}
		\abs{x-x_0}^\frac{\alpha+2}{\alpha+1}.
\end{align*}
It turns out that a similar result can be obtained for supersolutions to \eqref{eq:eigen_main_equation} by taking a suitable transformation. Indeed we may define
$\phi(r) := \Phi(x)$ with $r = \abs{x}$ to see that $\psi := \exp(\phi) > 0$ satisfies
\begin{align*}
	\frac{F^*(\nabla \psi(r), \D^2 \psi(r))}{\abs{\psi(r)}^\alpha \psi(r)}
	&= -\abs{\phi'(r)}^\alpha
	\left(
	\phi''(r)	\cQ + \frac{1}{r} \phi'(r) (\cF - \cQ) + \phi'(r)^2 \cQ
	\right)	\\
	&\leq -\abs{\phi'(r)}^\alpha
	\left(
	\phi''(r)	\cQ + \frac{1}{r} \phi'(r) (\cF - \cQ)
	\right)	\\
	&= K
\end{align*}
for $r > 0$, according to the preceding section.

\begin{proposition}	\label{prop:exp_cusp_solution}
	Let $U \subset \Omega$ be bounded, $x_0 \in U$, and $K \geq 0$ be any nonnegative constant. Furthermore let $a \in \R$ be any constant. Then the function
	\[
		\Psi(x) := \exp(\Phi(x))
	\]
	satisfies
	\[
		F^*( \nabla \Psi, \D^2 \Psi) \leq K \abs{\Psi}^\alpha \Psi
	\]
	in $U \setminus \{x_0\}$.
\end{proposition}
\begin{proof}
	Again, the proof only involves computation.
\end{proof}
\kskip

\begin{lemma}	\label{lem:eigen_comparison_with_radial}
	Let $U \subset \Omega$ be bounded, $x_0 \in U$ and $u \in \LSC{\overline{\Omega}}$ be a supersolution to \eqref{eq:eigen_main_equation}.
	Then
	\begin{align*}
		u(x) \geq \Psi(x)	\qquad\text{on } \partial(U \setminus\{x_0\})
		&&\implies&&
		u(x) \geq \Psi(x)	\qquad\text{in } U
	\end{align*}
	with
	\[
		\Psi(x) := \exp(\Phi(x))
	\]
	for $K = \inf_U f$.
\end{lemma}
\begin{proof}
	We assume there is some $\hat x \in U$ such that
	\begin{align*}
		u(x) \geq \Psi(x) \qquad\text{on } \partial(U \setminus\{x_0\})
		&&\text{but}&&
		u(\hat x) < \Psi(\hat x).
	\end{align*}
	Since $U$ is bounded, there is some ball of radius $R$ covering $U$. We define
	\begin{align*}
		\Phi_\eps(x) :=
		a &- \frac{\alpha+1}{\alpha+2}
		\left( \Thetac \left(\inf_U f - \eps\right) \right)^\frac{1}{\alpha+1}
		\abs{x-x_0}^\frac{\alpha+2}{\alpha+1}	\\
		&- \frac{\alpha+1}{\alpha+2}
		\left( \Thetac \eps \right)^\frac{1}{\alpha+1}
		R^\frac{\alpha+2}{\alpha+1}.
	\end{align*}
	Then we have $\Phi(x) \geq \Phi_\eps(x)$ so
	the function
	\[
		\Psi_\eps(x) := \exp(\Phi_\eps(x))
	\]
	satisfies
	\[
		u(x) \geq \Psi(x) \geq \Psi_\eps(x) \qquad\text{on } \partial(U \setminus\{x_0\})
	\]
	while still keeping
	\[
		u(\hat x) < \Psi_\eps(\hat x)
	\]
	and $\inf_U f - \eps \geq  0$ for $\eps > 0$ sufficiently small. Therefore we may assume that $x \mapsto u(x) - \Psi_\eps(x)$ attains a local minimum at $\hat x$. There
	\[
		F^*( \nabla \Psi_\eps(\hat x), \D^2 \Psi_\eps(\hat x))
		\leq \left(\inf_U f - \eps \right) \abs{\Psi_\eps(\hat x)}^\alpha \Psi_\eps(\hat x)
		< f(\hat x) \abs{\Psi_\eps(\hat x)}^\alpha \Psi_\eps(\hat x)
	\]
	holds, contradicting the assumption on $u$ being a supersolution to \eqref{eq:eigen_main_equation}.
\end{proof}
\kskip

\begin{lemma}[Hopf's lemma]
	\label{lem:eigen_hopf}
	Let $u \in \LSC{\overline{\Omega}}$ be a positive supersolution to \eqref{eq:eigen_main_equation}, $\Omega$ satisfy the interior sphere condition, and let $x_0 \in \partial \Omega$ be a boundary point with
	\begin{equation}	\label{eq:eigen_hopf_assumption}
		0 = u(x_0) < u(x)
	\end{equation}
	for all $x \in \Omega$. Then
	\[
		\limsup_{r \downarrow 0} \frac{u(x_0) - u(x_0 - r \mu)}{r} < 0
	\]
	for any $\mu \in \R^n$ with $\< \mu, \nu(x_0) \> > 0$
	with $\nu(x_0)$ denoting the outer normal at $x_0$.
\end{lemma}
\begin{proof}
	Since $\Omega$ satisfies the interior sphere condition there is some $R > 0$ such that $B_R(y_0) \subset \Omega$ and $\partial B_R(y_0) \cap \partial \Omega = \{x_0\}$ with $y_0 := x_0 - R \nu(x_0)$.
	Since $f$ is positive, $v := u + \eps$ is a supersolution to
	\[
		F(\nabla v, \D^2 v) = \eps \inf_\Omega f \abs{v}^\alpha v
	\]
	in $\Omega$ for $\eps > 0$ small. We may take $\eps > 0$ so small that
	\begin{equation}	\label{eq:eigen_hopf_equation1}
		v(y_0) \geq
		\eps \exp\left(\frac{\alpha+1}{\alpha+2}
		\left( \Thetac \left( \eps \inf_{B_R(y_0)} f \right) \right)^\frac{1}{\alpha+1}
		R^\frac{\alpha+2}{\alpha+1}\right).
	\end{equation}
	Then we define
	\begin{align*}
		\psi(x) :=
		\eps  \exp\left(\frac{\alpha+1}{\alpha+2}
		\left( \Thetac \left( \eps \inf_{B_R(y_0)} f \right) \right)^\frac{1}{\alpha+1} \times \right.\\\left.
		\left(R^\frac{\alpha+2}{\alpha+1} - \abs{x-y_0}^\frac{\alpha+2}{\alpha+1}\right)\right).
	\end{align*}
	Now we have
	\[
		v(y_0) \geq \psi(y_0)
	\]
	by \eqref{eq:eigen_hopf_equation1} and
	\[
		v(x) \geq \eps = \psi(x)
	\]
	on $\partial B_R(y_0)$ by \eqref{eq:eigen_hopf_assumption}. Therefore we have $v \geq \psi$ on $\partial (B_R(y_0) \setminus \{y_0\})$.
	Then Lemma \ref{lem:eigen_comparison_with_radial} implies $v \geq \psi$ in $B_R(y_0)$.
	\\
	
	Using $v(x_0) = \psi(x_0)$ and $x_0 - r \mu \in B_R(y_0)$ we conclude
	\[
		\frac{u(x_0) - u(x_0 - r \mu)}{r}
		= \frac{v(x_0) - v(x_0 - r \mu)}{r}
		\leq \frac{\psi(x_0) - \psi(x_0 - r \mu)}{r}
	\]
	for all $r > 0$ sufficiently small. Sending $r \downarrow 0$ we find
	\begin{align*}
		&\limsup_{r \downarrow 0} \frac{u(x_0) - u(x_0 - r \mu)}{r}	\\
		&\qquad\leq	\frac{\partial \psi}{\partial \mu}(x_0)	\\
		&\qquad=	- \eps^2 \left( \Thetac \left( \eps \inf_{B_R(y_0)} f\right) \, R \right)^\frac{1}{\alpha+1}	\< \nu(x_0), \mu \>\\
		&\qquad<	0,
	\end{align*}
	proving the assertion.
\end{proof}

We turn to showing that any positive solution $u$ to
\begin{align} \label{eq:eigen_main_equation_constant}
	F(\nabla u, \D^2 u) = \lambda \abs{u}^\alpha u
\end{align}
with vanishing boundary values, has the property that
\[
	w := -\log u
\]
is a convex function. Here $\lambda > 0$ denotes any positive eigenvalue. The procedure is the same as in the preceding  section. We first consider which problem $w$ is solving.

\begin{lemma}	\label{lem:eigen_transformed_equation}
	A function $u \in \USC{\Omega}$ is a positive subsolution to \eqref{eq:eigen_main_equation_constant} if and only if  $w 
	\in \LSC{\Omega}$ is a supersolution to
	\begin{align}	\label{eq:eigen_main_transformed}
	F(\nabla w, \D^2 w)
	=	-\abs{\nabla w}^{\alpha+2} \cQ - \lambda.
	\end{align}
\end{lemma}
\begin{proof}
	Again, this proof only involves computations.
\end{proof}

Since $w := -\log u$ degenerates at the boundary of $\Omega$ to $+\infty$, its convex envelope cannot be spanned by boundary points. Therefore we may omit the analogon to Lemma \ref{lem:convex_decomp} and proceed with showing that the convex envelope $w_{**}$ is a supersolution.

\begin{lemma}	\label{lem:eigen_convex_also_solution}
	Let $u \in \USC{\overline{\Omega}}$ be a positive subsolution to \eqref{eq:eigen_main_equation_constant} with $u = 0$ on $\partial \Omega$. Then
	$w_{**}$ is a supersolution to \eqref{eq:eigen_main_transformed} with $w_{**} = +\infty$ on $\partial \Omega$.
\end{lemma}

\begin{proof}
	According to \cite[Lemma 1]{alvarez1997convex} we have $w_{**} = w = +\infty$ on $\partial \Omega$. So we only have to show that $w_{**}$ is indeed a supersolution to \eqref{eq:eigen_main_transformed}. This means showing that
	\[
		F^*(q, X) +\abs{\nabla w_{**}}^{\alpha+2} \cQ + \lambda
		\geq	0
	\]
	for all $x \in \Omega$ and $(q,X) \in J_{\overline{\Omega}}^{2,-}w_{**}(x)$.\\
	Using \cite[Lemma 3]{alvarez1997convex} we only have to consider the case of $X \geq 0$.
	\\
	
	If $q = 0$ and $\alpha > 0$ we have
	\[
		F^*(q,X) = 0
	\]
	by \eqref{ass:boundedness}. So the assertion is obvious since all the other terms are nonnegative. So let us assume $q \neq 0$ or $\alpha = 0$.
	\\
	
	Since $w_{**}$ cannot be spanned by any boundary points, we can decompose $x$ in a convex combination of interior points $x_1, \ldots, x_k \in \Omega$ such that
	\begin{align*}
		\sum_{i=1}^{k} \mu_i x_i = x
	&&\text{and}&&
		\sum_{i=1}^{k} \mu_i w(x_i) = w_{**}(x)
	\end{align*}
	for some $\mu_i, \ldots, \mu_k > 0$ with $\sum_{i=1}^{k} \mu_i = 1$.
	Then by \cite[Proposition 1]{alvarez1997convex} for every $\eps > 0$ small enough, there are positive semidefinite matrices $X_1,\ldots,X_k \in \Symp$ such that
	\[
		X - \eps X^2
		\leq \left( \sum_{i=1}^{k} \mu_i X_i^{-1}\right)
		=: Y.
	\]
	Again, we may assume that the matrices $X_1, \ldots, X_k$ are positive definite. Otherwise we consider matrices
	\[
		\tilde X_i := X_i + \tfrac{1}{n} \Id
	\]
	and afterwards take the limit $n \to \infty$ as in
	\cite[p. 273]{alvarez1997convex}.
	Using that by assumption $w$ is a supersolution to \eqref{eq:eigen_main_transformed}, we find
	\[
		F^*(q, X_i) \geq
		-\abs{q}^{\alpha+2} \cQ - \lambda
	\]
	for $i = 1,\ldots, k$.	
	Since $X_i$ is positive definite and we are considering the case of $q \neq 0$ or $\alpha =0$, we have $F^*(q,X_i) < 0$ according to \eqref{ass:boundedness} which allows us to rearrange the preceding inequality to
	\[
		1 \leq - \frac{1}{F^*(q, X_i)} \left(\cQ \abs{q}^{\alpha+2} + \lambda \right)
	\]
	for $i = 1,\ldots, k$. By summation and rearranging terms again, we obtain
	\[
		\cQ \abs{q}^{\alpha+2} + \lambda \geq \left( \sum_{i=1}^{k} \mu_i \frac{1}{- F^*(q,X_i)} \right)^{-1}.
	\]
	Using that $F^*$ is degenerate elliptic and plugging in $X - \eps X^2 \leq Y$ we find
	\begin{align*}
	&F^*(q, X - \eps X^2) + \cQ \abs{q}^{\alpha+2} + \lambda	\\
	&\qquad\geq F^*(q, Y) + \cQ \abs{q}^{\alpha+2} + \lambda	\\
	&\qquad\geq F^*(q, Y) + \left( \sum_{i=1}^{k} \mu_i \frac{1}{- F^*(q,X_i)} \right)^{-1}.
	\end{align*}
	By Lemma \ref{lem:operator_concavity} we have
	\begin{align*}
	\sum_{i=1}^{k} \mu_i \frac{1}{-F^*(q,X_i)}
	&= \sum_{i=1}^{k} \mu_i \frac{1}{-F^*(q,(X_i^{-1})^{-1})}	\\
	&\leq \frac{1}{-F^*(q, \sum_{i=1}^k \mu_i X_i^{-1})}	\\
	&= \frac{1}{-F^*(q, Y)}
	\end{align*}
	so we may combine these two estimates to obtain
	\[
	F^*(q, X - \eps X^2) \geq 0.
	\]
	Sending $\eps \downarrow 0$ concludes the assertion.
\end{proof}

Just like Theorem \ref{th:concavity} we conclude our second main result.

\begin{theorem}
	Let $u \in C(\Omega)$ be a positive viscosity solution to \eqref{eq:eigen_main_equation_constant} with $u = 0$ on $\partial \Omega$ in a convex domain $\Omega$ that satisfies the interior sphere condition. Furthermore we assume that \eqref{ass:linearity}, \eqref{ass:boundedness}, \eqref{ass:scaling}, \eqref{ass:rotinvariance}, and \eqref{ass:operator_concavity} hold.
	Then $\log u$ is concave.
\end{theorem}

\begin{acknowledgement}
	\vfill
	I thank B. Kawohl for suggesting the subject of this paper and for helpful discussion.
\end{acknowledgement}

	\bibliography{power_and_log_concavity}
	\bibliographystyle{alpha}
\end{document}